\documentclass[12pt,reqno]{article}

\usepackage[usenames]{color}
\usepackage{amssymb}
\usepackage{amsmath}
\usepackage{amsthm}
\usepackage{amsfonts}
\usepackage{amscd}
\usepackage{graphicx}

\usepackage[colorlinks=true,
linkcolor=webgreen,
filecolor=webbrown,
citecolor=webgreen,
hypertexnames=false]{hyperref}

\definecolor{webgreen}{rgb}{0,.5,0}
\definecolor{webbrown}{rgb}{.6,0,0}

\usepackage{color}
\usepackage{fullpage}
\usepackage{float}

\usepackage{latexsym}

\newcommand{\seqnum}[1]{\href{https://oeis.org/#1}{\rm \underline{#1}}}

\usepackage{amssymb,latexsym}	
\usepackage{booktabs}

\newcommand{\Q}{{\mathbb Q}}

\newcommand{\N}{{\mathbb N}}

\newcommand{\floor}[1]{\left\lfloor#1\right\rfloor}
\newcommand{\arXiv}[2]{\href{https://arxiv.org/abs/#1}{\texttt{arXiv:#1 [#2]}}}

\DeclareMathOperator{\Res}{Res}
\DeclareMathOperator{\Syl}{Syl}

\begin{document}

%\theoremstyle{plain}
%\numberwithin{equation}{section}
%\newtheorem{thm}{Theorem}[section]
%\newtheorem{theorem}[thm]{Theorem}
%\newtheorem{lemma}[thm]{Lemma}
%\newtheorem{corollary}[thm]{Corollary}

\theoremstyle{plain}
\newtheorem{theorem}{Theorem}
\newtheorem{corollary}[theorem]{Corollary}
\newtheorem{lemma}[theorem]{Lemma}
\newtheorem{proposition}[theorem]{Proposition}

\theoremstyle{definition}
\newtheorem{definition}[theorem]{Definition}
\newtheorem{example}[theorem]{Example}
\newtheorem{conjecture}[theorem]{Conjecture}

\theoremstyle{remark}
\newtheorem{remark}[theorem]{Remark}

\begin{center}
\vskip 1cm{\Large\bf 
Binomial Convolutions for Rational Power Series
}
\vskip 1cm
Ira M. Gessel\footnote{Supported by a grant from the Simons Foundation (\#427060, Ira Gessel).}\\
Department of Mathematics \\
Brandeis University \\
Waltham, MA 02453 \\USA\\
\href{mailto:gessel@brandeis.edu}{\tt gessel@brandeis.edu} \\
\ \\
Ishan Kar\\
College of Letters and Science \\ 
University of California\\
Bekeley, CA 94720\\ 
USA \\ 
\href{mailto:(ishankar@berkeley.edu}{\tt ishankar@berkeley.edu}
\end{center}

\vskip .2 in

\begin{abstract}
The binomial convolution of two sequences $\{a_n\}$ and $\{b_n\}$ is the sequence whose $n$th term is 
$\sum_{k=0}^{n} \binom{n}{k} a_k b_{n-k}$. If $\{a_n\}$ and $\{b_n\}$ have rational generating functions then so does their binomial convolution. We discuss an efficient method, using resultants, for computing  this rational  generating function and give several examples involving Fibonacci and tribonacci numbers and related sequences. We then describe a similar method for computing Hadamard products of rational generating functions. 
Finally we describe two additional methods for computing binomial convolutions and Hadamard products of rational power series, one using symmetric functions and one using partial fractions.
\end{abstract}

\section{Introduction}
The ordinary convolution of two sequences $\{\alpha_n\}$ and $\{\beta_n\}$ is the sequence $\{\gamma_n\}$ defined by $\gamma_n = \sum_{k=0}^n \alpha_k \beta_{n-k}$. Ordinary convolutions are closely related to ordinary generating functions: if $\{\gamma_n\}$ is the ordinary convolution of $\{\alpha_n\}$ and $\{\beta_n\}$ then 
\begin{equation*}
\biggl(\sum_{n=0}^\infty \alpha_n x^n\biggr)\biggl( \sum_{n=0}^\infty \beta_n x^n\bigg) = \sum_{n=0}^\infty \gamma_n x^n.
\end{equation*}

The \emph{binomial convolution} of two sequences $\{a_n\}$ and $\{b_n\}$ is the sequence $\{c_n\}$ defined by $c_n = \sum_{k=0}^{n} \binom{n}{k} a_k b_{n-k}$. The \emph{exponential generating function} 
$\sum_{n=0}^\infty c_n {x^n}/{n!}$ for $\{c_n\}$
is the product of the exponential generating functions for $\{a_n\}$ and $\{b_n\}$. So binomial convolutions of sequences with simple exponential generating functions are easily dealt with.

We are concerned here with binomial convolutions of sequences with  rational ordinary
generating functions. 
Let us define an operation $\odot$ on formal power series in $x$ with coefficients in a field of characteristic $0$ by
\begin{equation} \label{e-oper}
    \sum_{n=0}^\infty a_nx^n \odot \sum_{n=0}^\infty b_nx^n = \sum_{n=0}^\infty c_nx^n,
\end{equation}
where $c_n = \sum_{k=0}^n \binom{n}{k} a_k b_{n-k}$.
In other words, \eqref{e-oper} is equivalent to 
\begin{equation}\label{e-gen}
    \biggl( \sum_{n=0}^\infty a_n \frac{x^n}{n!} \biggl) \biggl( \sum_{n=0}^\infty b_n \frac{x^n}{n!} \biggl) =
    \sum_{n=0}^\infty c_n \frac{x^n}{n!}.
\end{equation}

We call $A(x)\odot B(x)$ the \emph{binomial product} of $A(x)$ and $B(x)$. (This operation is sometimes called the \emph{Hurwitz product} \cite{abm, fliess} or \emph{shuffle product} \cite{bacher}.)
It follows from the equivalence of \eqref{e-oper} and \eqref{e-gen} that $\odot$ is associative, with identity element 1, and that the inverse with respect to $\odot$ of $1/(1-\alpha x)$ is $1/(1+\alpha x)$.
As we will see in Theorem \ref{t-rbc}, if $A(x)$ and $B(x)$ are rational, then so is $A(x)\odot B(x)$. 
We will discuss several methods for computing binomial products of rational power series and give examples. 

In Section \ref{s-result} we describe a method for computing binomial products of rational power series efficiently, using resultants. 
In Section \ref{s-ex} we give several examples of binomial convolution formulas involving Fibonacci, tribonacci, and Perrin numbers.

We discuss in Section \ref{s-Hadamard}
a similar method, using resultants, for computing Hadamard products of rational power series. Then in Section \ref{s-sym}  we discuss a different method, using symmetric functions, for computing binomial products and Hadamard products of rational power series, and in Section \ref{s-pf} we discuss another method for binomial and Hadamard products, using partial fractions to compute constant terms of Laurent series.

\subsection{Computing binomial convolutions}
\label{s-cbc}
For rational power series with denominators of degree 2, it is not difficult to compute binomial convolutions directly. For example, for the Fibonacci numbers $F_n$, with 
\begin{equation*}
\sum_{n=0}^\infty F_n x^n = \frac{x}{1-x-x^2},
\end{equation*}
we have the explicit formula $F_n = (\alpha^n - \beta^n)/\sqrt5$, where $\alpha=(1+\sqrt5)/2$ and 
$\beta=(1-\sqrt5)/2$. Thus since $\alpha+1=\alpha^2$ and $\beta+1=\beta^2$, we have
\begin{align*}
\sum_{k=0}^n\binom{n}{k}F_k &= \frac{1}{\sqrt5}\sum_{k=0}^n \binom{n}{k}(\alpha^k - \beta^k)\\
  &=\bigl((\alpha+1)^n -(\beta+1)^n\bigr)/\sqrt5\\
  &=(\alpha^{2n}-\beta^{2n})/\sqrt5 = F_{2n},
\end{align*}
as is well known. (Note that our computation does not  use the explicit values of $\alpha$ and $\beta$, but only the quadratic equation that they satisfy.)
In the same way, we can prove the binomial convolution
\begin{equation}
\label{e-cb1}
\sum_{k=0}^n\binom nk F_k F_{n-k} = \frac15(2^n L_n -2),
\end{equation}
where $L_n$ is the Lucas number $F_{n-1}+F_{n+1}$,
as shown by Church and Bicknell \cite{church}.

More challenging are sequences whose generating functions have a denominator of degree~3. One well-known example is the tribonacci sequence, defined by
$T_{-1}=T_0=0$, $T_1=1$, and  $T_n= T_{n-1}+T_{n-2}+T_{n-3}$ for $n\ge2$, with generating function
\begin{equation*}
\sum_{n=0}^\infty T_n x^n = \frac{x}{1-x-x^2-x^3}.
\end{equation*}
(There are several different conventions for the index of the first nonzero tribonacci number. We start with $T_1$, following Komatsu \cite{koma3}, but some authors start with $T_0$ or $T_2$.)
The tribonacci numbers are sequence \seqnum{A000073} in the On-Line Encyclopedia of Integer Sequences (OEIS) \cite{oeis}.

Komatsu \cite{koma} 
found a formula, given in  \eqref{e-komatsu3}, for 
the numbers
\begin{equation}
    \sum_{k=0}^{n} \binom{n}{k} \ T_k T_{n-k} \label{e-komatsu1}
\end{equation}
by using the exponential generating function for $T_n$, which involves the zeros of a cubic polynomial. We will give a simpler proof of Komatsu's formula, and a generalization, in Section \ref{s-komatsu}.

Prodinger  \cite{prod} proved Komatsu's formula by first computing the binomial product
\begin{align*}
\frac{x}{1-x-x^2-x^3}\odot  \frac{x}{1-x-x^2-x^3}
  &= \frac{1}{11} \biggl(\frac{1+x
+10x^2}{1-2x-4x^2-8x^3} - \frac{1+x-8x^2}{1-2x+2x^3}
\biggl).
\end{align*}
He used the following approach to computing binomial products of rational power series. Suppose we know the generating functions $\sum_{n=0}^\infty a_n x^n =A(x)$ and $\sum_{n=0}^\infty b_n x^n = B(x)$.
Then a straightforward computation gives
\begin{equation}
\label{e-prod1}
\sum_{n=0}^\infty x^n\sum_{k=0}^n \binom{n}{k} a_k \sum_{m=k}^\infty y^m b_{m-k}
  =\frac{1}{1-x}A\left(\frac{xy}{1-x}\right) B(y).
\end{equation}
The binomial convolution $\sum_{k=0}^n \binom{n}{k} a_k b_{n-k}$ is the coefficient of $x^ny^n$ in \eqref{e-prod1}, so the binomial product $A(x)\odot B(x)$ may be obtained by extracting the diagonal from \eqref{e-prod1}. For $A(x)$ and $B(x)$ rational, Prodinger did this using Hautus and Klarner's residue method \cite{hau}. We will discuss a more efficient method for computing the diagonal of \eqref{e-prod1} in Section \ref{s-pf}.

Ekhad and Zeilberger \cite{ekha} used another efficient method for computing binomial products of rational power series by solving a system of linear equations, using the fact that a proper rational power series in $x$ with a denominator of degree $d$ is  determined by the coefficients of $x^i$ for $0\le i \le 2d$.  (Recall that a proper rational function is one in which the degree of the numerator is less than the degree of the denominator.)

Cerlienco, Mignotte, and Piras \cite[Section A IV 1]{cmp} sketched another approach, using determinants, for computing binomial and Hadamard products.

Our main method for computing binomial products is based on resultants of polynomials. If the denominator of $A(x)$ is $\prod_i(1-\alpha_i x)$ and the denominator of $B(x)$ is $\prod_j (1-\beta_j x)$ then $A(x)\odot B(x)$  is a rational function  with denominator (not necessarily in lowest terms) $\prod_{i,j}(1-\alpha_i \beta_j x)$, and this product can be computed from the coefficients of the denominators of $A(x)$ and $B(x)$ in terms of a resultant,  with no need to factor these denominators into linear factors. 
We use a similar resultant method to compute the denominator of the Hadamard product of two rational generating functions in Section \ref{s-Hadamard}.

Alecci, Barbero, and Murru \cite[Theorem 5]{abm} have also observed that resultants can be used to compute binomial products.

\section{On rational power series over a field}
We first prove a useful fact that clarifies the status of rational power series over a field. Another proof was given by Klazar and  Horsk\'{y} \cite[Theorem 3]{kh}. The result can also be derived using properties of determinants; for example, Equation (10) of Cerlienco, Mignotte, and Piras \cite{cmp} yields a determinantal formula for the denominator of a rational power series in terms of its coefficients.

\begin{theorem}
\label{t-fields}
Suppose that 
\begin{equation*}
A(x)=\sum_{n=0}^\infty a_n x^n = \frac{p(x)}{q(x)},
\end{equation*}
where each $a_n$ is in a field $F$, and $p(x)$ and $q(x)$ are polynomials with coefficients in an extension field $G$ of $F$. Then there exist polynomials $P(x)$ and $Q(x)$ with coefficients in $F$ such that
\begin{equation*}
A(x) = \frac{P(x)}{Q(x)}.
\end{equation*}
\end{theorem}

In order to prove Theorem \ref{t-fields}, we first prove two lemmas from linear algebra.

\begin{lemma}
\label{l-lind} Let $F\subseteq G$ be fields. If a subset $W\!$ of $F^n$ is linearly independent over $F$ then it is linearly independent over $G$.
\end{lemma}
\begin{proof}
Let $W\subset F^n$ be linearly independent over $F$. Since $W$ is finite, we may let the elements of $W$ be $(w_{i,1}, w_{i,2},\dots, w_{i,n})$, for $i=1,\dots, |W|$, and consider the matrix $M=(w_{i,j})_{i=1,\dots, |W|; j=1,\dots, n}$. We can convert $M$ to a matrix $M'$ in row-echelon form by operations that preserve the span of its rows over both $F$ and $G$, and the dimension of the row space of $M'$ over both $F$ and $G$
is the number of nonzero rows. \end{proof}

\begin{lemma}
\label{l-fields}
Let $F\subseteq G$ be fields. Suppose that $S$ is a set of vectors in $F^n$ and that there is a nonzero vector $t$ in $G^n$ such that $s\cdot t=0$ for all $s\in S$, where $\cdot$ is the usual dot product defined by $(s_1,\dots, s_n)\cdot (t_1,\dots, t_n) = s_1t_1+\cdots + s_n t_n$. Then there is a nonzero vector $u\in F^n$ such that $s\cdot u=0$ for all $s\in S$.
\end{lemma}
\begin{proof} 
The hypothesis implies that the dimension of the span of $S$ over $G$ is less than $n$. Since the dimension of the span of $S$ (over either $F$ or $G$) is the size of a largest linearly independent subset, Lemma \ref{l-lind} implies that the dimension of $S$ over $F$ is also less than $n$.

The lemma then follows from a basic fact of linear algebra:  If $K$ is a field and 
 $U$ is a set of vectors in $K^n$ then the dimension of the span of $U$ over $K$ is less than $n$ if and only if there is a nonzero vector $t\in K^n$ such that $u\cdot t=0$ for every $u\in U$.
\end{proof}

\begin{proof}[Proof of Theorem \ref{t-fields}]
 Let $q(x) = q_0+q_1x +\cdots+q_m x^m$. Since $q(x)\sum_{n=0}^\infty a_n x^n$ is a polynomial in $x$, it follows that for some $N$,
\[a_i q_m + a_{i+1} q_{m-1}+\cdots + a_{i+m} q_0  = 0\]
for all $i\ge N$. Then by Lemma \ref{l-fields}, there exist $Q_0,\dots, Q_m$ in $F$, not all 0, such that
\[a_i Q_m + a_{i+1} Q_{m-1}+\cdots + a_{i+m} Q_0  = 0\] for all $i\ge N$. Thus
\begin{equation*}
(Q_0+Q_1x+\cdots+Q_m x^m)A(x)
\end{equation*}
is a polynomial in $x$ with coefficients in $F$, so $A(x)$ is a quotient of polynomials with coefficients in $F$.
\end{proof}

In what follows, we will be working with power series with coefficients in a field $F$, and we will derive formulas for these power series that express them as quotients of polynomials with coefficients in an extension field of $F$. Theorem \ref{t-fields} guarantees that we can express these power series as quotients of polynomials with coefficients in $F$. Moreover, Theorem \ref{t-fields} allows us to use the term ``rational power series over the field $F$" to mean both a rational power series with coefficients in $F$ and a power series which is a quotient of polynomials with coefficients in $F$.

\section{Rationality of binomial products}

We now prove that the binomial product of two rational power series is rational. There are more direct proofs but our approach proves some useful formulas along the way. For other proofs see Fliess \cite[Proposition 3]{fliess}, Bacher \cite[Proposition 3.2]{bacher}, and  Alecci, Barbero, and Murru \cite[Theorem 5]{abm}. Another proof, pointed out by Alin Bostan,  that the binomial product of rational power series is rational follows from the fact that $\sum_{n=0}^\infty a_n x^n$ is rational if and only if $\sum_{n=0}^\infty a_n x^n/n!$ is a linear combination of series of the form $x^i e^{\alpha x}$, as these are clearly closed under multiplication.

\begin{lemma}
\label{l-bprod}
For all nonnegative integers $j$ and $k$ and all $\alpha, \beta\in F$ we have
\begin{equation}
\label{e-bprod}
\frac{x^j}{(1-\alpha x)^{j+1}}\odot \frac{x^k}{(1-\beta x)^{k+1}} = \binom{j+k}{j}\frac{x^{j+k}}{\bigl(1-(\alpha+\beta)x\bigr)^{j+k+1}}.
\end{equation}
\end{lemma}
\begin{proof} 
For any $\alpha\ne0$ we define $f_n(\alpha, k)$  for $n,k\in \N$ (where $\N$ is the set of nonnegative integers) by 
$f_n(\alpha, k) = \alpha^{n-k}\binom{n}{k}$. We define $f_n(0,k)$ to be the limit
$\lim_{\alpha\to0}f_n(\alpha,k)$, so 
\begin{equation*}
f_n(0,k) = \begin{cases}
1,&\text{if $n=k$}\\
0,&\text{otherwise.}
\end{cases}
\end{equation*}
(More precisely,
for fixed $n$ and $k$, $f_n(\alpha,k)$ is a polynomial in $\alpha$. If $n<k$ then this polynomial is identically~0 and if $n\ge k$ this polynomial is a constant times a nonnegative power of $\alpha$; in all cases $f_n(0,k)$ is obtained by setting $\alpha=0$ in this polynomial.)
It is easy to check that
\begin{equation*}
\sum_{n=0}^\infty f_n(\alpha, k) x^n =\frac{x^k}{(1-\alpha x)^{k+1}}
\end{equation*}
and 
\begin{equation*}
\sum_{n=0}^\infty  f_n(\alpha, k) \frac{x^n}{n!} = \frac{x^k}{k!} e^{\alpha x}
\end{equation*}
for all $\alpha$, including $\alpha=0$.
Thus 
\begin{equation*}
\biggl(\sum_{n=0}^\infty  f_n(\alpha, j) \frac{x^n}{n!}\biggr)
\biggl(\sum_{n=0}^\infty  f_n(\beta, k) \frac{x^n}{n!}\biggr)
 = \binom{j+k}{j}\frac{x^{j+k}}{(j+k)!} e^{(\alpha+\beta) x},
\end{equation*}
and \eqref{e-bprod} follows from the equivalence of \eqref{e-oper} and \eqref{e-gen}.
\end{proof}
We note a useful consequence of \eqref{e-bprod}.

\begin{corollary}
For any power series $A(x)$ and any $\beta\in F$, we have
\begin{equation*}
A(x) \odot \frac{1}{1-\beta x} = \frac{1}{1-\beta x}A\left(\frac {x}{1-\beta x}\right).
\end{equation*}
\end{corollary}
\begin{proof}
By linearity, it is sufficient to prove the formula for $A(x) = x^j$, which is the case $\alpha=0$, $k=0$ of \eqref{e-bprod}.
\end{proof}

\begin{theorem} \label{t-rbc}
Let $A(x)$ and $B(x)$ be rational power series over $F$. We may express $A(x)$ 
as
\begin{equation*}
A(x) = \frac{R(x)}{\prod_{i=1}^m (1-\alpha_i x)},
\end{equation*}
where $R(x)$ is a polynomial of degree less than $m$ and the $\alpha_i$ lie in some extension field of $F$,  but they need not be distinct nor nonzero.
We may express $B(x)$ similarly as 
\begin{equation*}
B(x) = \frac{S(x)}{\prod_{j=1}^n (1-\beta_j x)}.
\end{equation*}
Then 
\begin{equation}
\label{e-AB1}
 A(x)\odot B(x) = 
 \frac{T(x)}{\prod_{i=1}^m\prod_{j=1}^n\bigl(1-(\alpha_i+\beta_j)x\bigr)}.
\end{equation}
for some polynomial $T(x)$ of degree at most $mn-1$.
\end{theorem}

\begin{proof}
We first find the usual partial fraction decomposition of $A(1/x)/x$,
\[A(1/x)/x = U(x) +\sum_{i=1}^{M} \sum_{k=0}^{e_i}\frac{A_{ik}}{(x-a_i)^{k+1}}\]
where $U(x)$ is a polynomial,  the $a_i$  and $A_{ik}$ lie in some extension field of $F$,  the $a_i$ are distinct (one of them may be 0), and the $e_i$ are nonnegative integers.
Replacing $x$ with $1/x$ and dividing by $x$ gives us 
\begin{equation}
\label{e-pfA}
A(x) = \frac{U(1/x)}{x}+\sum_{i=1}^{M}\sum_{k=0}^{e_i} \frac{A_{ik}x^{k}}{(1-a_i x)^{k+1}}.
\end{equation}
Since $A(x)$ has a power series expansion, $U(x)$ must be $0$. 
(We note that $A(x)$ is proper if and only if $a_i\ne 0$ for all $i$.)

Similarly we have the partial fraction expansion
\[B(x) = \sum_{j=1}^{N}\sum_{l=0}^{f_j} \frac{B_{j l}x^{l}}{(1-b_j x)^{l+1}}.\] 
Then
\begin{align}
    A(x)\odot B(x) &= 
    \sum_{i=1}^M\sum_{j=1}^N \sum_{k=0}^{e_i}\sum_{l=0}^{f_j}
     A_{ik} B_{j l}\binom{k+l}{k}\frac{x^{k+l}}{\bigl(1-(a_i+b_j)x\bigr)^{k+l+1}}
     \text{ by   \eqref{e-bprod}}\notag \\
     &=\sum_{i=1}^M\sum_{j=1}^N \frac{P_{ij}(x)}{\bigl(1-(a_i+b_j)x\bigr)^{e_i+f_j+1}}\label{e-AB3}
\end{align}
where $P_{ij}(x)$  is a polynomial of degree at most $e_i+f_j$.

It follows from \eqref{e-AB3} that 
\begin{equation}
\label{e-AB4}
A(x) \odot B(x) = \frac{Q(x)}{\prod_{i=1}^M\prod_{j=1}^N \bigl(1-(a_i+b_j)x\bigr)^{e_i+f_j+1}}
\end{equation}
for some polynomial $Q(x)$ of degree less than
$\sum_{i=1}^M \sum_{j=1}^N (e_i+f_j+1)$.
In particular, $ A(x)\odot B(x)$ is rational. 

Now choose a sequence $\alpha_1, \dots, \alpha_m$, where  $m=\sum_{i=1}^M (e_i+1)$, consisting of $e_i$ occurrences of $a_i$ for each $i$, so that $\prod_{i=1}^m (1-\alpha_i x)=\prod_{i=1}^M(1-a_i x)^{e_i+1}$ and similarly choose $\beta_1,\dots, \beta_n$, where $n=\sum_{j=1}^N (f_j+1)$, so that 
$\prod_{j=1}^n (1-\beta_j x)=\prod_{j=1}^N(1-b_j x)^{f_j+1}$. Then 
\begin{equation*}
\prod_{i=1}^m\prod_{j=1}^n\bigl(1-(\alpha_i+\beta_j)x\bigr)
  =\prod_{i=1}^M\prod_{j=1}^N\bigl(1-(a_i+b_j)x\bigr)^{(e_i+1)(f_j+1)}.
\end{equation*}
Since $(e_i+1)(f_j+1)- (e_i+f_j+1) = e_i f_j\ge0$, multiplying the numerator and denominator of the right side of  \eqref{e-AB4} by $\prod_{i=1}^M\prod_{j=1}^N\bigl(1-(a_i+b_j)x\bigr)^{e_i f_j}$ gives the desired representation \eqref{e-AB1}.
\end{proof}

We may restate Theorem \ref{t-rbc} in a less elegant but more computationally practical way by separating out the $\alpha_i$ and $\beta_j$ that are equal to 0. 

\begin{corollary}
\label{c-main}
Let
\begin{equation*}
A(x) = \frac{R(x)}{\prod_{i=1}^m (1-\alpha_i x)},
\end{equation*}
and
\begin{equation*}
B(x) = \frac{S(x)}{\prod_{j=1}^n (1-\beta_j x)},
\end{equation*}
where $R(x)$ and $S(x)$ are polynomials and
the $\alpha_i$ and $\beta_j$ are all nonzero. Let \[u=\max\bigl(\deg R(x)+1-m,0\bigr)\textnormal{ and }v=\max\bigl(\deg S(x)+1-n,0\bigl).\]
Then
\begin{equation}
\label{e-AB5}
 A(x)\odot B(x) = 
 \frac{T(x)}{\bigl(\prod_{i=1}^m(1-\alpha_i x)\bigr)^v\bigl( \prod_{j=1}^n(1-\beta_j x)\bigr)^u
 \prod_{i=1}^m\prod_{j=1}^n\bigl(1-(\alpha_i+\beta_j)x\bigr)}.
\end{equation}
for some polynomial $T(x)$ of degree less than $(u+m)(v+n)$.
\end{corollary}

Note that $A(x)\odot B(x)$ need not be proper even if $A(x)$ and $B(x)$ are; for example by Lemma \ref{e-bprod} we have
\begin{equation*}
\frac{x^{j}}{(1-x)^{j+1}}\odot \frac{x^k}{(1+x)^{k+1}}=\binom{j+k}{j}x^{j+k}.
\end{equation*}

As is clear from the proof of Theorem \ref{t-rbc}, the denominator in lowest terms of $A(x)\odot B(x)$ will in many cases be a proper divisor of that given by \eqref{e-AB1} or \eqref{e-AB5}, but these formulas, together with Theorem \ref{t-res1}, which allows us to compute $\prod_{i=1}^m\prod_{j=1}^n\bigl(1-(\alpha_i+\beta_j)x\bigr)$ without factoring the denominators of $A(x)$ and $B(x)$, are efficient enough in all of our applications.

One important special case is when $A(x)$ and $B(x)$ have the same denominator. If we assume, for simplicity, that $A(x)$ and $B(x)$ are proper and have no repeated zeros, then the partial fraction expansion \eqref{e-AB3} shows that $A(x)\odot B(x)$ can be written with denominator  $\prod_{i=1}^n (1-2\alpha_i x) \prod_{1\le i<j\le n}\bigl(1-(\alpha_i+\alpha_j) x\bigr)$.

As a first example, with proper rational functions and no repeated factors, we have 
\begin{equation*}
\frac{x}{(1-x)(1-2x)}\odot \frac{x}{(1-3x)(1-5x)} = \frac{x^2(2-11x)}{(1-4x)(1-5x)(1-6x)(1-7x)}.
\end{equation*}
An example with an improper rational function is 
\begin{equation*}
\frac{x^3}{1-x}\odot \frac{1}{1-2x} = \frac{x^3}{(1-2x)^3(1-3x)}.
\end{equation*}
With the notation of Corollary \ref{c-main}, here we have $m=1$, $n=1$, $u=3$, and $v=0$.
Finally, a more complicated example is 
\begin{equation*}
\frac{x^2}{(1-x)^2} \odot \frac{x^2}{(1-2x)^2} = \frac{x^4(6-30 x +49 x^{2}-27 x^{3})}{(1-x)^2(1-2x)^2(1-3x)^3}.
\end{equation*}
Here the denominator given by Corollary \ref{c-main} is $(1-x)^2(1-2x)^2(1-3x)^4$ but because of the multiple factors in the denominators $(1-x)^2$ and $(1-2x)^2$, the actually denominator is a proper divisor of this product.

An alternative approach to binomial products of improper rational power series is to decompose a rational function as a polynomial plus a proper rational function. For proper rational power series we can take all the $\alpha_i$ and $\beta_j$ in Theorem \ref{t-rbc} to be nonzero; equivalently, in Corollary \ref{c-main} we have $u=v=0$. We then need to compute binomial products of polynomials with rational power series. 
By linearity, it is sufficient to compute $x^m\odot A(x)$ in terms of $A(x)$. The next lemma tells us how to do this.
\begin{lemma}
Let $A(x)$ be a power series in $x$. Then for any nonnegative integer $m$ we have
\begin{equation}
\label{e-poly}
x^m\odot A(x) = \frac{x^m}{m!}\frac{d^m\ }{dx^m}\bigl(x^m A(x)\bigr).
\end{equation}
\end{lemma}
\begin{proof}
By linearity, it is sufficient to prove \eqref{e-poly} in the case $A(x)=x^n$. In this case we have
\begin{equation*}
\frac{x^m}{m!}\frac{d^m\ }{dx^m}\bigl(x^m A(x)\bigr) 
  =\frac{x^m}{m!}\frac{d^m\ }{dx^m}x^{m+n}=\binom{m+n}{m} x^{m+n}=x^m \odot A(x).
  \qedhere
\end{equation*}
\end{proof}

\section{Resultants}
\label{s-result}
\subsection{The resultant of two polynomials}
Given two polynomials \[A(x) = \sum_{i=0}^m a_i x^i = a_m\prod_{i=1}^m (x-\alpha_{i})\] and \[B(x) = \sum_{j=0}^n b_j x^{k} = b_n\prod_{j=1}^n (x-\beta_{j}),\] 
their resultant with respect  to the variable $x$ may be 
defined by
\begin{equation} \label{e-res}
   \Res(A(x),B(x),x) = a_m^n b_n^m \prod_{i,j} (\alpha_i-\beta_j).
\end{equation}
It is well known \cite{syl} that $\Res(A(x),B(x),x)$ can be computed as a determinant.
\begin{lemma}
\label{l-res}
The resultant can be computed by 
\begin{equation}
\label{e-Sylv}
\Res(A(x),B(x),x)=\det\Syl(A(x),B(x)),
\end{equation}
where $\Syl(A(x),B(x))$ is the Sylvester matrix of size $(n+m) \times (n+m)$,
\[
\Syl(A(x),B(x))=\left[\phantom{\begin{matrix}\alpha_0\\ \ddots\\\alpha_0\\\beta_0\\ \ddots\\\beta_0 \end{matrix}}
\right.\hspace{-1.5em}
\underbrace{\begin{matrix}
a_m & \cdots & a_0 & \\
\ddots & & \ddots & \\
& a_m & \cdots & a_0 \\
b_n & \cdots & b_0 & \\
\ddots & & \ddots & \\
&b_n & \cdots & b_0
\end{matrix}}_{m+n}
\hspace{-1.5em}
\left.\phantom{\begin{matrix}\alpha_0\\ \ddots\\\alpha_0\\\beta_0\\ \ddots\\\beta_0 \end{matrix}}\right]\hspace{-1em}
\begin{tabular}{l}
$\left.\lefteqn{\phantom{\begin{matrix} \alpha_0\\ \ddots\\ \alpha_0\ \end{matrix}}}\right\}n$\\
$\left.\lefteqn{\phantom{\begin{matrix} \beta_0\\ \ddots\\ \beta_0\ \end{matrix}}} \right\}m$
\end{tabular}
\]
\end{lemma}
Note that unlike  formula \eqref{e-res}, the determinant formula \eqref{e-Sylv} expresses the resultant in terms of the coefficients of $A$ and $B$.

\subsection{Computing the binomial product with resultants}
\label{s-comput}

We now describe how to compute the product $\prod_{i=1}^m\prod_{j=1}^n\bigl(1-(\alpha_i+\beta_j)x\bigr)$ in Corollary \ref{c-main} without factoring the denominators.

\begin{theorem}
\label{t-res1}
Let $U(x) = \prod_{i=1}^m (1-\alpha_i x)$ and $V(x) = \prod_{j=1}^n (1-\beta_j x)$. Then
\begin{equation*}
\prod_{i=1}^m \prod_{j=1}^n \bigl(1- (\alpha_i+\beta_j) x\bigr) 
     =(-1)^{mn}\Res\left((1-y)^m U\Bigl(\frac{x}{1-y}\Bigr), y^nV\Bigl(\frac{x}{y}\Bigr),y\right).
\end{equation*}
\end{theorem}
\begin{proof}
We have
\begin{equation*}
(1-y)^m U\Bigl(\frac{x}{1-y}\Bigr)=\prod_{i=1}^m (1-y-\alpha_i x)=(-1)^m\prod_{i=1}^m \bigl(y-(1-\alpha_i x)\bigr)
\end{equation*}
and
\begin{equation*}
y^n V(x/y) = \prod_{j=1} (y-\beta_j x),
\end{equation*}
so the result follows from \eqref{e-res}.
\end{proof}

Corollary \ref{c-main} and Theorem \ref{t-res1} together give us a procedure for computing binomial products of rational power series that can be efficiently implemented on a computer algebra system. By Corollary \ref{c-main} we can express the binomial convolution $A(x)\odot B(x)$ as a quotient of polynomials $T(x)/D(x)$, where $D(x)$ may be computed explicitly with the help of Theorem \ref{t-res1}, and we have a bound on the degree of $T(x)$, $\deg T(x)\le d$. Then $T(x)$ can be computed from  \[T(x) = \bigl(A(x)\odot B(x)\bigr) D(x)\]
by computing the coefficients up to $x^d$ on the right side.

As a simple example, let us compute in detail the binomial product of the generating functions for the Fibonacci and Pell sequences. 
The generating function for the Fibonacci sequence is $A(x) = {x}/(1-x-x^2)$ and that for the Pell sequence is $B(x) = x/(1-2x-x^2)$.  Since both are proper, by Corollary \ref{c-main} we know that the binomial product $A(x)\odot B(x)$ may be expressed as $T(x)/D(x)$ where $T(x)$ has degree at most 3 and $D(x)$ is a polynomial of degree~4 computed by Theorem \ref{t-res1}. 

With the notation of Theorem \ref{t-res1}, we have $m=n=2$,  $U(x)= 1-x-x^2$, and $V(x) = 1-2x-x^2$, so 
\begin{equation*}
(1-y)^2 U\left(\frac{x}{1-y}\right) = y^2 +(x-2)y-x^2-x-1
\end{equation*}
and
\begin{equation*}
y^2 V(x/y) = y^2-2xy - x^2.
\end{equation*}

The Sylvester matrix  is 
\[\begin{bmatrix}
1 & x-2 & -x^2-x+1 & 0\\
0 & 1 & x-2 & -x^2-x+1 \\
1 & -2x & -x^2 & 0 \\
0 & 1 & -2x & -x^2
\end{bmatrix}
\]
with determinant $D(x)=1-6x+7x^2+6x^3 -9x^4$.
We compute directly that
\begin{equation*}
A(x)\odot B(x) = 2x^2+9x^3+40x^4+\cdots
\end{equation*}
So the numerator is \[T(x) = (2x^2+9x^3+40x^4+\cdots)(1-6x+7x^2+6x^3-9x^4) = 2x^2-3x^3.\]
Thus
\begin{equation*}
\frac{x}{1-x-x^2}\odot \frac{x}{1-2x-x^2} = \frac{2x^2-3x^3}{1-6x+7x^2+6x^3-9x^4}.
\end{equation*}

\section{Examples}
\label{s-ex}

\subsection{Fibonacci numbers}
We first discuss some binomial convolution formulas for Fibonacci numbers. We omit the details of the computations of binomial products, which are done by the computer using the procedure described in Section \ref{s-comput}, together with partial fraction expansion.

Recall that the Fibonacci numbers have the generating function
\begin{equation}
\label{e-Fgf}
\sum_{n=0}^\infty F_n x^n = \frac{x}{1-x-x^2}
\end{equation}
and the Lucas numbers $L_n=F_{n-1} + F_{n+1}$ have the generating function
\begin{equation*}
\sum_{n=0}^\infty L_n x^n = \frac{2-x}{1-x-x^2}
\end{equation*}
Let us first prove the identity of Church and Bicknell \cite{church} mentioned in Section \ref{s-cbc}. 
(See also Prodinger \cite{prod}.)
We find that 
\begin{equation*}
 \frac{x}{1-x-x^2}\odot \frac{x}{1-x-x^2}= \frac{2x^2}{1-3 x -2 x^{2}+4 x^{3}}
 =\frac15\left(\frac{2-2x}{1-2x-4x^2}-\frac{2}{1-x}\right).
\end{equation*}
Since 
\begin{equation*}
\frac{2-2x}{1-2x-4x^2}=\frac{2-(2x)}{1-(2x) - (2x)^2} = \sum_{n=0}^\infty 2^n L_n x^n,
\end{equation*}
we have
\begin{equation}
\label{e-cb}
\sum_{k=0}^n\binom nk F_k F_{n-k} = \frac15(2^n L_n -2).
\end{equation}
This identity may be generalized in many ways. We describe one generalization, based on the multisected Fibonacci generating functions \cite{hoggatt}
\begin{equation}
\label{e-multiF}
\sum_{n=0}^\infty F_{pn+q}x^n = \frac{F_q + (-1)^qF_{p-q}x}{1-L_px+(-1)^px^2}
\end{equation}
and 
\begin{equation}
\label{e-multiL}
\sum_{n=0}^\infty L_{pn+q}x^n = \frac{L_q - (-1)^q L_{p-q}x}{1-L_px+(-1)^px^2}.
\end{equation}
Here $p$ and $q$ may be arbitrary integers, and the Fibonacci and Lucas numbers are extended to negative integer indices by $F_{-n}=(-1)^{n-1}F_n$ and $L_{-n}=(-1)^n L_n$.
Letting $f(x) = x/(1-x-x^2)$, we find that

\begin{equation}
\label{e-fab}
5f(ax)\odot f(bx)=
\frac{2-\left(a +b \right) x}{1-\left(a +b \right) x -\left(a +b \right)^{2} x^{2}}
-\frac{2-\left(a +b \right) x}{1-\left(a +b \right) x-\left(a^{2}-3 a b +b^{2}\right) x^{2}}.
\end{equation}
The first term on the right in \eqref{e-fab} is 
\begin{equation*}
\sum_{n=0}^\infty (a+b)^n L_n x^n.
\end{equation*}
We would like to choose $a$ and $b$ so that the denominator in the second term on the right in \eqref{e-fab} is $1-L_p(cx) + (-1)^p (cx)^2$ for some $p$ and $c$, so that we can apply \eqref{e-multiF} or \eqref{e-multiL}. A calculation that we omit suggests that we should take $a=F_{p-1}$ and $b=F_{p+1}$ for some~$p$. Once we have these values for $a$ and $b$, it is easy to verify the result: We have $a+b=F_{p-1}+F_{p+1}=L_p$ and 
\begin{align*}
a^2 -3ab+b^2 &= (b-a)^2 -ab \\&= (F_{p+1}-F_{p-1})^2 - F_{p+1}F_{p-1}=F_p^2 - F_{p+1}F_{p-1}\\&=(-1)^{p-1}
\end{align*}
by Cassini's identity for Fibonacci numbers.
Thus 
\begin{equation*}
5f(F_{p-1}x)\odot f(F_{p+1}x) =\frac{2-L_p x}{1-L_p x-L_p^2 x^2}
-\frac{2-L_p x}{1-L_p x+(-1)^p x^2}
\end{equation*}
which by \eqref{e-Fgf} and \eqref{e-multiL} is equal to 
\begin{equation*}
\sum_{n=0}^\infty L_p^n F_n x^n -\sum_{n=0}^\infty L_{pn}x^n.
\end{equation*}
Thus
\begin{equation}
\label{e-gcb}
\sum_{k=0}^n \binom{n}{k}F_{p-1}^kF_{p+1}^{n-k}F_kF_{n-k} = \frac{1}{5}(L_p^n L_n -L_{pn}).
\end{equation}
Church and Bicknell's identity \eqref{e-cb} is the case $p=0$ of \eqref{e-gcb}.

As another example, we have
\begin{align*}
\frac{x}{1-x-x^2}\odot \frac{1}{1-x^2} &= \frac{x-x^2}{1-2 x -3 x^{2}+4 x^{3}-x^{4}}\\
  &=\frac12\left( \frac{x}{1+x-x^2}+\frac{x}{1-3x+x^2} \right)\\
  &=\frac12 \sum_{n=0}^\infty (-1)^{n-1}F_n x^n +\frac12 \sum_{n=0}^\infty F_{2n}x^n,
\end{align*}
where we have used the case $p=2,q=0$ of \eqref{e-multiF},
so
\begin{equation*}
\sum_{k=0}^{\floor{n/2}} \binom{n}{2k}F_{n-2k}=\tfrac12\bigl((-1)^{n-1}F_n + F_{2n}\bigr).
\end{equation*}

The squares of the Fibonacci numbers have the well-known generating function 
\begin{equation}
\label{e-fibsq}
\sum_{n=0}^\infty F_n^2 x^n = \frac{x-x^2}{1-2x-2x^2+x^3}
\end{equation}
which we will rederive as \eqref{e-fibsq2}.
We have
\begin{align*}
\frac{x-x^2}{1-2x-2x^2+x^3} \odot \frac{10}{1-5x^2} &= 
\frac{x -3 x^{2}-2 x^{4}-6 x^{5}}{1-4 x -15 x^{2}+50 x^{3}+35 x^{4}-114 x^{5}+36 x^{6}}\\
  &=\frac{2-3x}{1-3x+x^2}-\frac{4+4x}{1+2x-4x^2}+\frac{2-3x}{1-3x-9x^2}\\
  &=\sum_{n=0}^\infty L_{2n}x^n +\sum_{n=0}^\infty (-2)^{n+1} L_n x^n +\sum_{n=0}^\infty 3^n L_n
\end{align*}
and we find the binomial convolution
\begin{equation*}
10\sum_{k=0}^{\floor{n/2}}\binom{n}{2k}5^k F_{n-2k}^2 = L_{2n}+(3^n+(-2)^{n+1})L_n.
\end{equation*}

\subsection{Second order recurrent sequences}
Let $g(x) = (c+dx)/(1-ax-bx^2)$. We find that
\begin{equation*}
g(x)\odot g(x) =\frac{p}{1-ax} + \frac{q+rx}{1-2ax-4bx^2},
\end{equation*}
where $p$, $q$, and $r$ are certain rational functions of $a$, $b$, $c$, and $d$, too complicated to be worth writing out here. If we put some restrictions on the parameters, we get a much nicer formula.

\begin{theorem}
\label{t-G}
Let \[g(x) = \sum_{n=0}^\infty G_n x^n = \frac{2-ax}{1-ax-bx^2}.\]
Then \[g(x)\odot g(x) = \frac{2}{1-ax}+g(2x),\]
so 
\begin{equation*}
\sum_{k=0}^n \binom{n}{k} G_k G_{n-k}=2a^n +2^n G_n.\qedhere
\end{equation*}
\end{theorem}

Special cases of Theorem \ref{t-G} are $a=b=1$ (Lucas numbers; this identity was given by Church and Bicknell \cite{church}), $a=1, b=2$ (Jacobsthal-Lucas numbers, \seqnum{A014551}), and $a=2, b=1$ (companion Pell numbers, \seqnum{A002203}).

\subsection{Tribonacci numbers}
\label{s-komatsu}
As noted in the introduction, the tribonacci numbers $T_n$ may be defined by the generating function
\begin{equation*}
\sum_{n=0}^\infty T_n x^n = \frac{x}{1-x-x^2-x^3}.
\end{equation*}

Komatsu \cite{koma3} (see also Komatsu and Li \cite{koma1} and Komatsu \cite{koma})
gave the binomial convolution formula
\begin{equation}
\sum_{k=0}^n\binom{n}{k} T_k T_{n-k} = \frac{1}{22}\biggl(2^n T_{n}^{(2,3,10)}+2\sum_{k=0}^n \binom{n}{k} (-1)^k T_{k}^{(-1,2,7)}\biggr).
\label{e-komatsu3}
\end{equation}
Here the numbers $T_n^{(s_0, s_1, s_2)}$ satisfy the tribonacci recurrence $T_n= T_{n-1}+T_{n-2}+T_{n-3}$  for $n\ge 3$ with initial values $T_0^{(s_0, s_1, s_2)}=s_0$, 
$T_1^{(s_0, s_1, s_2)}=s_1$, and $T_2^{(s_0, s_1, s_2)}=s_2$, so the ordinary tribonacci numbers are $T_n = T_n^{(0,1,1)}$. 
It is easy to check that the generating function for $T_n^{(s_0, s_1, s_2)}$ is
\begin{equation*}
\sum_{n=0}^\infty T_n^{(s_0, s_1, s_2)}x^n=\frac{a+bx+cx^2}{1-x-x^2-x^3},
\end{equation*}
where $a=s_0$, $b=s_1-s_0$, and $c=s_2-s_1-s_0$. (Equivalently, $T_n^{(s_0, s_1, s_2)}=aT_{n+1}+bT_n +cT_{n-1}$ with these values of $a$, $b$, and $c$.)

We now derive Komatsu's formula with binomial products.
Let $t(x) =x/(1-x-x^2-x^3)$ be the tribonacci generating function. Then we compute
\begin{align}
t(x)\odot t(x) &= \frac{2x^2(1-x-x^2 -2x^3)}{1-4 x +2 x^{3}+12 x^{4}-8 x^{5}-16 x^{6}}\notag\\
  &=\frac{1+x+10x^2}{11(1-2x-4x^2-8x^3)}-\frac{1+x-8x^2}{11(1-2x+2x^3)}\label{e-koma1},
\end{align}
as found by another method by Prodinger \cite{prod}.
The first term in \eqref{e-koma1} is 
\begin{equation*}
\frac{1}{22}\sum_{n=0}^\infty 2^n T_n^{(2,3,10)}x^n.
\end{equation*}
For the second term in \eqref{e-koma1} we can check easily that 
\begin{equation*}
\frac{1+x-8x^2}{1-2x+2x^3}=\frac{1}{1-x}\odot \frac{1+3x -6x^2}{1+x-x^2+x^3}
\end{equation*}
and that 
\[\frac{1+3x -6x^2}{1+x-x^2+x^3} = \sum_{n=0}^\infty T_n^{(1,-2,-7)}(-x)^n 
  = -\sum_{n=0}^\infty T_n^{(-1,2,7)}(-x)^n,\]
proving Komatsu's formula.

There is nothing really special about tribonacci numbers here;  something similar will hold in general for generating functions with a cubic denominator:

\begin{theorem}
\label{t-komatsu}
Let $r(x)$ and $s(x)$ be proper rational functions with the same denominator $D(x)=1+Ax+Bx^2+Cx^3$, where $C\ne0$. Suppose that $D(x)=(1-\alpha_1 x)(1-\alpha_2 x)(1-\alpha_3 x)$ where $2\alpha_i \ne \alpha_j +\alpha_k$ for  $i,j,k\in\{1,2,3\}$ not all equal. 
Then for some polynomials $u(x)$ and $v(x)$ of degree at most $2$,
\begin{equation}
\label{e-cubic}
r(x)\odot s(x) = \frac{u(x)}{D(2x)} + \frac{1}{1-Ax}\odot \frac{v(x)}{D(-x)}.
\end{equation}

\end{theorem}
\begin{proof}
The denominator of $r(x)\odot s(x)$ may be factored as $D_1(x)D_2(x)$, where
$D_1(x) = (1-2\alpha_1 x)(1-2\alpha_2 x)(1-2\alpha_3 x)$ and $D_2(x)=\bigl(1-(\alpha_1+\alpha_2)x\bigr)\bigl(1-(\alpha_1+\alpha_3)x\bigr)\bigl(1-(\alpha_2+\alpha_3)x\bigr)$. The condition on the $\alpha_i$ implies that $D_1(x)$ and $D_2(x)$ are relatively prime, so we have the
partial fraction expansion, 
\[r(x)\odot s(x) = \frac{u(x)}{D_1(x)} + \frac{w(x)}{D_2(x)}
\]
for some polynomials $u(x)$ and $w(x)$ of degree at most 2. 
Clearly $D_1(x) = D(2x)$.
Since $\alpha_1+\alpha_2+\alpha_3=-A$, we have
\begin{align*}
D_2(x) &= \bigl(1+(\alpha_1+A)x)(1+(\alpha_2+A)x)(1+(\alpha_3+A)x).
\end{align*}
Therefore $1/(1+Ax)\odot w(x)/D_2(x)$ has denominator $D(-x)$
and so may be expressed as $v(x)/D(-x)$ for some $v(x)$ of degree at most 2; i.e.,
\begin{equation*}
\frac{1}{1+Ax}\odot \frac{w(x)}{D_2(x)}=\frac{v(x)}{D(-x)}
\end{equation*}  Since $1/(1-Ax)$ and $1/(1+Ax)$
are inverses with respect to $\odot$, this implies that
\begin{equation*}
 \frac{w(x)}{D_2(x)}=\frac{1}{1-Ax}\odot\frac{v(x)}{D(-x)}.
 \qedhere
\end{equation*} 
\end{proof}

Note that the condition on the $\alpha_i$ is equivalent to $D(x)$ having no repeated zeros and  $D_1(x)$ and $D_2(x)$ being relatively prime, and this will be easy to check in any particular example.  If $D_1(x)$ are $D_2(x)$ are not relatively prime we will have a similar but slightly different partial fraction expansion.  

\subsection{Perrin numbers}
Another interesting third-order recurrent sequence is the Perrin sequence (\seqnum{A001608}) with generating function
\begin{equation*}
P(x) = \sum_{n=0}^\infty P_n x^n = \frac{3 - x^2}{1 - x^2 - x^3}.
\end{equation*}
Here the parameter $A$ of Theorem \ref{t-komatsu} is 0, so \eqref{e-cubic} will simplify. However, the result turns out to be even nicer than we have any reason to expect.

We have 
\begin{align*}
P(x) \odot P(x) &= 3\frac{3-11 x^{2}-15 x^{3}+4 x^{4}+4 x^{5}}{1-5 x^{2}-7 x^{3}+4 x^{4}+4 x^{5}-8 x^{6}}\\
  &=\frac{3-4 x^{2}}{1-4 x^{2}-8 x^{3}} +\frac{6-2 x^{2}}{1-x^{2}+x^{3}}\\
  &=P(2x) + 2P(-x).
\end{align*}
Thus we have the unexpectedly simple formula
\begin{equation*}
\sum_{k=0}^n \binom{n}{k} P_k P_{n-k}= \bigl(2^n+2(-1)^n\bigr)P_n.
\end{equation*}

More generally, if we set $Q(x) = (3-x^2)/(1-x^2 -ax^3)$, where $a$ is arbitrary, then $Q(x)\odot Q(x) = Q(2x) + 2Q(-x)$.

Surprisingly, the Jacobsthal number generating function $J(x)=x/(1-x-2x^2)$ satisfies a very similar formula,
$3J(x)\odot J(x) = J(2x) +2J(-x)$.

We give one more curious identity, involving a fourth-order recurrent sequence. Let
\begin{equation*}
R(x) = \frac{1-2x^3}{1-8x^3+4x^4}.
\end{equation*}
Then 
\begin{equation*}
R(x) \odot R(x) = \frac{1}{4}\bigl(R(2x) + P(4x^2)\bigr),
\end{equation*}
where $P(x)$ is the Perrin sequence generating function.

\section{The Hadamard product of rational power series}  \label{s-Hadamard}
The Hadamard product \cite{ikar} of the power series  
$A(x)$ $= \sum\limits_{n = 0}^\infty a_nx^n$ and
$B(x) = \sum\limits_{n = 0}^\infty b_nx^n$ is defined by    
\begin{equation}\label{hp}
    A(x) \ast  B(x) = \sum\limits_{n = 0}^\infty a_nb_nx^n.
 \end{equation}

As we shall see, Hadamard products have some similarity to binomial products. We give here a brief account of the computation of Hadamard products of rational power series using resultants. Another approach to Hadamard products of rational power series using determinants has been given by Potekhina and Tolovikov \cite{pote2014, pote2017}.

It is well known and easy to prove (see, e.g., Stanley \cite[Theorem 4.1.1]{stan}) that  $f(x)$ is a proper rational power series if and only if the coefficient of $x^n$ in $f(x)$ is of the form $\sum_i P_i(n) \alpha_i^n$ where each $P_i(n)$ is a polynomial in $n$ and each $\alpha_i$ is nonzero. Clearly the product of two functions of this type is also of this type, so the Hadamard product of proper rational power series is a proper rational power series, and it follows easily that the Hadamard product of any two rational power series is rational.

There is an interesting explicit formula, somewhat analogous to Lemma \ref{l-bprod}, for a special case of Hadamard products.
We have
\begin{align}
\frac{x^i}{(1-a x)^{m+1}}*\frac{x^j}{(1-b x)^{n+1}}
    &=\sum_{k=0}^\infty \binom{m+k-i}{k-i}\binom{n+k-j}{k-j} a^{k-i}b^{k-j}x^k\notag\\
    &=\frac{\sum_{k=0}^\infty\binom{m+j-i}{k-i}\binom{n+i-j}{k-j}a ^{k-i}b^{k-j}x^k}
       {(1-ab x)^{m+n+1}}.\label{e-had1}
\end{align}
The numerator of \eqref{e-had1} may be expressed as a Hadamard product; it is $x^i(1+a x)^{m+j-i}*x^j(1+b x)^{n+i-j}$.

Equation \eqref{e-had1}
 is equivalent to a classical formula of Euler for hypergeometric series \cite[p.~2, Equation (2)]{bailey}. See also Kar \cite[Theorem 4.1]{ikar} for another proof of the case $i=j=0$.
This formula
actually holds for all $m$ and $n$, 
where the binomial coefficient $\binom{a}{b}$ is defined to be $a(a-1)\cdots (a-b+1)/b!$ if $b$ is a nonnegative integer, and is $0$ otherwise. 

If $m$ and $n$ are nonnegative integers then the sum in the numerator of \eqref{e-had1} has only finitely many nonzero terms. In particular, if $m$ and $n$ are nonnegative integers,  $i\le m+j$, and $j\le n+i$, which is always the case when \eqref{e-had1} represents a Hadamard product of proper rational functions, then the nonzero terms in the numerator range from $k=\max(i,j)$ to $k=\min(n+i, m+j)$.

It follows from \eqref{e-had1} (or otherwise) that if $A(x)$ is a proper rational function with denominator $(1-ax)^{m+1}$ and $B(x)$ is a proper rational function with denominator $(1-bx)^{n+1}$ then $A(x)*B(x)$ is a proper rational function that may be written with denominator $(1-abx)^{(m+1)(n+1)}$.
Thus we have the following  analogue of Theorem \ref{t-rbc}.

\begin{theorem}
\label{t-had} Let $A(x)$ be a rational power series with denominator $\prod_{i=1}^m(1-\alpha_i x)$
and let $B(x)$ be a  rational power series with denominator $\prod_{j=1}^n(1-\beta_j x)$. Then $A(x)*B(x)$ is a  rational power series that may be written with denominator $\prod_{i=1}^m\prod_{j=1}^n (1-\alpha_i\beta_j x)$. Moreover, if $A(x)$ and $B(x)$ are proper then so is $A(x)*B(x)$. \qed
\end{theorem}

We now give the analogue for Hadamard products of Theorem \ref{t-res1}.

\begin{theorem}
\label{t-hadt}
Let $U(x) = \prod_{i=1}^m (1-\alpha_i x)$ and $V(x) = \prod_{j=1}^n (1-\beta_j x)$. Then
\[\prod_{i=1}^m \prod_{j=1}^n (1-\alpha_i\beta_jx) = (-1)^{mn} \Res\bigl( U(y), y^n V(x/y),y\bigr). \]
\end{theorem}
\begin{proof}
We have 
\begin{equation*}
U(y)=(-1)^m \prod_{i=1}^m \alpha_i\,\prod_{i=1}^m (y - \alpha_i^{-1})
\end{equation*}
and 
\begin{equation*}
y^n V(x/y) = \prod_{j=1}^m (y-\beta_j x).
\end{equation*}
Thus 
\begin{align*}
 \Res\bigl( U(y), y^n V(x/y),y\bigr) 
   &= (-1)^{mn}\biggl(\prod_{i=1}^m \alpha_i\biggr)^{\!n} 
      \prod_{i=1}^m\prod_{j=1}^n (\alpha_i^{-1}-\beta_j x) \\
   &= (-1)^{mn}\prod_{i=1}^m\prod_{j=1}^n (1-\alpha_i\beta_j x).\qedhere
\end{align*}
\end{proof}

As with binomial products, the numerator  of $A(x)*B(x)$ can be found by multiplying the denominator with the first $mn$ terms of the infinite series expansion of $A(x) \ast B(x)$.

As an example, we use this method to find the Hadamard product of  \[A(x) = \frac{x}{1-ax-bx^2}\] and \[B(x) = \frac{x}{1-cx-dx^2}.\] 
Examples of well-known second-order recurrent sequences with generating functions of this type are Fibonacci (\seqnum{A000045}), Pell (\seqnum{A000129}), and Jacobsthal (\seqnum{A001045}).

To compute $A(x)*B(x)$ we take 
$U(x) = 1-ax-bx^2=(1-\alpha_1 x)(1-\alpha_2 x)$ and $V(x) = 1-cx-dx^2=(1-\beta_1 x)(1-\beta_2 x)$ in Theorem 
\ref{t-hadt}.
For the denominator of $A(x)*B(x)$ we have
\begin{align}
\prod_{i=1}^2\prod_{j=1}^2 (1-\alpha_i\beta_j x) &= \Res(1-ay-by^2, y^2-cxy -dx^2,y)\notag\\
  &=1-acx-(a^2d+bc^2+2bd)x^2-abcdx^3+b^2d^2x^4.\label{e-den1}
\end{align}
Multiplying the denominator \eqref{e-den1} by the initial terms of
\[ A(x) \ast B(x) = x+acx^2+(a^2+b)(c^2+d)x^3+(a^3+2ab)(c^3+2cd)x^4+\cdots\]
we find the numerator $x-bdx^3$. So
\begin{equation}
\label{e-h1}
A(x) \ast B(x) = \frac{x-bdx^3}{1-acx-(a^2d+bc^2+2bd)x^2-abcdx^3+b^2d^2x^4}.
\end{equation}
A combinatorial proof of \eqref{e-h1} was given by Shapiro \cite{shap} and further combinatorial proofs of similar Hadamard product identities were given by Kim \cite{kim, kim2}.

As a special case of \eqref{e-h1} we have the well-known generating function for the squares of the Fibonacci numbers used in Section \ref{s-ex},
\begin{equation}
\label{e-fibsq2}
\sum_{n=0}^\infty F_n^2 x^n = \frac{x-x^3}{1-x -4 x^{2}-x^{3}+x^{4}}=\frac{x-x^2}{1-2 x -2 x^{2}+x^{3}}.
\end{equation}

Frontczak,  Goy,  and Shattuck \cite{fgsb,fgsa} give several formulas for Hadamard products of  two or three second-order recurrent sequences that can be easily be derived by our methods.

\section{Symmetric functions}
\label{s-sym}
In this section we briefly describe an alternative approach to computing denominators for binomial products and Hadamard products of rational power series. This method has been used by Dvornicich and Traverso \cite{dt}  to compute the resultants corresponding to these denominators. Bostan, Flajolet, Salvy, and Schost  \cite{bfss} reformulated this approach with a view to computational efficiency, obtaining a ``nearly optimal" algorithm for computing these resultants. 

Let
$U(x) =\prod_{i=1}^m (1-\alpha_i x)$ and $V(x) = \prod_{j=1}^n (1-\beta_j x)$.
We want to express
$
\prod_{i,j}\bigl(1-(\alpha_i+\beta_j)x\bigr)
$
and $\prod_{i,j}(1-\alpha_i\beta_j x)$
in terms of the coefficients of $U(x)$ and $V(x)$.

Let $e_k(u_1, \dots, u_n)$ denote the $i$th elementary symmetric polynomial in $u_1, \dots, u_n$,
\begin{equation*}
e_k(u_1,\dots, u_n) = \sum_{1\le i_1<i_2<\dots<i_k\le n}u_{i_1}u_{i_2}\dots u_{i_n},
\end{equation*}
with $e_0(u_1,\dots, u_n)=1$.
Then the coefficient of $x^k$ in $\prod_{i=1}^n (1-u_i x)$ is equal to \[(-1)^k e_k(u_1,\dots, u_n),\]
and $e_k(u_1, \dots, u_n) = 0$ for $k>n$.
Thus to find the denominator polynomials for the binomial and Hadamard products, it is  sufficient to express $e_k(\alpha_1 + \beta_1,\dots , \alpha_i + \beta_j,\dots, \alpha_m+\beta_n)$ and $e_k(\alpha_1\beta_1,\dots, \alpha_i\beta_j,\dots, \alpha_m\beta_n)$ in terms of 
the $e_k(\alpha)$ and $e_k(\beta)$, where for any symmetric polynomial $f$, $f(\alpha)$ means $f(\alpha_1, \dots, \alpha_m)$ and similarly for $f(\beta)$. 

To do this we use the power sum symmetric polynomials
$p_j(u_1,\dots, u_n) = u_1^j+\cdots + u_n^j.$
They are related to the elementary symmetric polynomials by 
the generating function formula
\begin{equation}
\label{e-pe}
\sum_{k=0}^\infty e_k z^k= \exp\biggl(\sum_{j=1}^\infty (-1)^{j-1} \frac{p_j}{j} z^j\biggr),
\end{equation}
where we are omitting the arguments to $e_k$ and $p_j$,
which upon differentiating with respect to $z$ yields Newton's recurrence
\begin{equation}
\label{e-newton}
ke_{k}=\sum _{i=1}^{k}(-1)^{i-1}e_{k-i}p_{i},
\end{equation}
from which the elementary and power sum symmetric polynomials can easily be determined from each other. (With a computer algebra system it may be easier, though less efficient, to use \eqref{e-pe} directly.) We will also need the value  $p_0(u_1,\dots, u_n) = n$.

Let us write $p_k(\alpha\odot \beta)$ for  \[p_k(\alpha_1 + \beta_1,\dots , \alpha_i + \beta_j,\dots,\alpha_m+\beta_n)\] and $p_k(\alpha*\beta)$ for \[p_k(\alpha_1\beta_1,\dots, \alpha_i\beta_j,\dots, \alpha_m\beta_n).\]
It is easy to see that 
\begin{equation}
\label{e-p*}
p_k(\alpha*\beta) =p_k(\alpha)p_k(\beta),
\end{equation}
and for $p_k(\alpha\odot\beta)$ we have
\begin{align}
p_k(\alpha\odot\beta) &= \sum_{i,j}(\alpha_i + \beta_j)^k
  = \sum_{i,j}\sum_{l=0}^k \binom{k}{l}\alpha_i^l\beta_j^{k-l}\notag\\
  &=\sum_{l=0}^k\binom{k}{l}p_l(\alpha)p_{k-l}(\beta).\label{e-pcomp}
\end{align}

We can then compute $p_k(\alpha)$ and $p_k(\beta)$ for $k$ from 1 to $mn$ from the coefficients of $U(x)$ and $V(x)$ using \eqref{e-newton}, then compute $p_k(\alpha\odot\beta)$ and $p_k(\alpha *\beta)$ from \eqref{e-p*} and \eqref{e-pcomp}, and finally compute $e_k(\alpha\odot\beta)$ and $e_k(\alpha *\beta)$ from \eqref{e-newton}.

As an example, we consider again the Fibonacci and Pell sequences; $A(x)=x/(1-x-x^2)$ and $B(x) = x/(1-2x-x^2)$.
Here we have $e_1(\alpha)=1$, $e_2(\alpha) = -1$, $e_1(\beta) = 2$, and $e_2(\beta) = -1$. 
Using \eqref{e-pe} or \eqref{e-newton} we compute $p_n(\alpha)$, which are the Lucas numbers, and $p_n(\beta)$, which are the companion Pell numbers (\seqnum{A002203}), for $n$ from 0 to 4. We then compute $p_n(\alpha*\beta)$ and $p_n(\alpha\odot\beta)$ from \eqref{e-p*} and \eqref{e-pcomp}, and finally we compute 
$e_n(\alpha*\beta)$ and $e_n(\alpha\odot\beta)$ from \eqref{e-pe} and \eqref{e-newton}.
The values that occur in this computation are listed in Table \ref{tab-1}.

\begin{table}[h!]
\begin{equation*}
\begin{array}{crrrrr}\toprule
n:&0 &  1 & 2 & 3 & 4\\ \midrule
e_n(\alpha)& 1&1&-1&0&0 \\
e_n(\beta)&1&2&-1&0&0 \\
p_n(\alpha)&2&1&3&4&7\\
p_n(\beta)&2&2&6&14&34\\
p_n(\alpha*\beta)&4& 2&18 & 56 &238\\
p_n(\alpha \odot\beta)&4&6&22&72&278\\
e_n(\alpha*\beta)&1&2&-7&2&1\\
e_n(\alpha\odot\beta)&1&6&7&-6&-9\\    
\bottomrule
\end{array}
\end{equation*}
\caption{Fibonacci and Pell sequence computation}
\label{tab-1}
\end{table}
Thus the denominator of $A(x)*B(x)$ is $1-2x-7x^2 -2x^3 +x^4$ and the denominator of $A(x)\odot B(x)$ is $1-6x+7x^2 +6x^3-9x^4$.

We note that the generating function $A(x)*B(x) = (x-x^3)/(1-2x-7x^2-2x^3+x^4)$ can be found in the OEIS \cite{oeis} at \seqnum{A001582}. See also Mez\H o \cite[Remark 9]{mezo}.

\section{Partial fractions}
\label{s-pf}
Here we describe  briefly another method, using partial fractions, for computing Hadamard and binomial products of rational power series. This method was used by Xin \cite[Section 1-2]{xin} for Hadamard products.

We start with a simple example, the Hadamard product of the generating functions for Fibonacci and Pell numbers. 
Let $A(x) = x/(1-x-x^2)=\sum_{n=0}^\infty F_n x^n$ and $B(x) = x/(1-2x-x^2)=\sum_{n=0}^\infty P_n x^n$.
Then 
\begin{equation}
\label{e-diag1}
f=A(t)B(x/t) = \sum_{m,n=0}^\infty F_m P_n x^n t^{m-n}.
\end{equation}
With series like \eqref{e-diag1} involving infinitely many negative powers of variables, we must be careful about the power series ring in which we are working. Here we are in the ring $\Q((t))[[x]]$ of formal power series in $x$ with coefficients that are Laurent series in $t$. So although elements of this ring may have arbitrary negative powers of $t$, the coefficient of any power of $x$ involves only finitely many negative powers of $t$. 

The constant term in $t$ in \eqref{e-diag1} is 
\begin{equation*}
\sum_{n=0}^\infty F_n P_n x^n =A(x)*B(x).
\end{equation*}
To compute it, we start with
the partial fraction expansion  of $A(t)B(x/t)$ as a rational function of $t$,
\begin{multline}
\label{e-pfr1}
\qquad
\frac{1}{1-t-t^2} \frac{x -x^{3}+\left(x^{3}+2 x^{2}\right) t}{1-2 x -7 x^{2}-2 x^{3}+x^{4}}+
\frac{1}{t^2-2xt-x^2}\frac{x^3-x^{5}+\left(x^{3}+2 x^{2}\right) t}{1-2 x -7 x^{2}-2 x^{3}+x^{4}}\\
=\frac{1}{1-2 x -7 x^{2}-2 x^{3}+x^{4}}\bigl( R+S\bigr),
\qquad
\end{multline}
where 
\begin{equation*}
R = \frac{x -x^{3}+\left(x^{3}+2 x^{2}\right) t}{1-t-t^2}
\end{equation*}
and 
\begin{equation*}
S=\frac{(x^3-x^5)t^{-2}+(2x^2+x^3)t^{-1}}{1-2xt^{-1}-x^2 t^{-2}}.
\end{equation*}

We want the coefficient of $t^0$ in \eqref{e-pfr1}. Note that as elements of $\Q((t))[[x]]$,  $R$ contains only nonnegative powers of $t$ and $S$ contains only negative powers of $t$. Thus the constant term in $t$ of $A(t)B(x/t)$ may be obtained by deleting $S$ from \eqref{e-pfr1} and then setting $t=0$, yielding
$A(x) *B(x) = (x-x^3)/(1-2x-7x^2 -2x^3 + x^4)$.

We can look at the partial fraction approach in a slightly different way, which shows how it is  related to resultants. With $A(x)$ and $B(x)$ as above, 
to find the partial fraction expansion in $t$ of $A(t)B(x/t)$, we need to find two polynomials $L(t)$ and $M(t)$ in $t$, each of degree at most~1, with coefficients that are rational functions of $x$,  such that
\begin{equation*}
\frac{L(t)}{1-t-t^2}+\frac{M(t)}{t^2-2xt-x^2}=A(t)B(x/t) = \frac{xt^2}{(1-t-t^2)(t^2-2xt-x^2)}.
\end{equation*}
Equivalently,
\begin{equation}
\label{e-uv1}
L(t)(-x^2-2xt+t^2) + M(t)(1-t-t^2) = xt^2.
\end{equation}
If we set $L(t) = a+bt$ and $M(t) = c+dt$, then equating coefficients of powers of $t$ in \eqref{e-uv1} gives the equivalent system
\begin{equation}
\label{e-uv2}
\begin{bmatrix}
-x^2 & -2x & 1 & 0\\
0&-x^2 & -2x & 1\\
1 & -1 & -1 & 0\\
0& 1 & -1 & -1
\end{bmatrix}
\begin{bmatrix}
a\\ b\\ c\\ d
\end{bmatrix}
=\begin{bmatrix}
0 \\0 \\ x \\ 0 
\end{bmatrix}
\end{equation}

Note that this matrix is essentially the Sylvester matrix for $-x^2-2xt+t^2$ and $1-t-t^2$, as polynomials in $t$. The constant term in $t$ of $A(t)B(x/t)$ is obtained by setting $t=0$ in $L(t)=a+bt$ so we may obtain this coefficient by solving for $a$ in \eqref{e-uv2}. By Cramer's rule, the denominator of $a$, as a rational function of $x$, will be the determinant of the matrix in \eqref{e-uv2}, so we see that this approach to Hadamard products is related to that of Section \ref{s-Hadamard}.

Our next theorem shows that the partial fraction method works in general for Hadamard products.   For simplicity we assume that our rational power series are proper.

\begin{theorem}
\label{t-pf}
Let $A(x)$ and $B(x)$ be proper rational power series with coefficients in a field $F$. Then   $A(t)B(x/t)$, as a rational function of $t$, has a partial fraction expansion of the form $U(x,t) + V(x,t)$ in which, as elements of the Laurent series field $F((t))((x))$,  $U(x,t)$ has only nonnegative powers of $t$ and $V(x,t)$ has only negative powers of $t$. Thus $A(x)*B(x)$ is obtained by setting $t=0$ in $U(x,t)$. 
\end{theorem}

\begin{proof}
Let $A(x)= N_A(x)/D_A(x)$ and $B(x) = N_B(x)/D_B(x)$, where $N_A(x)$, $D_A(x)$, $N_B(x)$, and $D_B(x)$ are polynomials and $D_B(x)$ has degree $n$. Then 
\begin{equation}
\label{e-NB}
A(t)B(x/t) = \frac{N_A(t)\cdot t^n N_B(x/t)}{D_A(t)\cdot t^n D_B(x/t)}
\end{equation}
We first note that $D_A(t)$ and $t^nD_B(x/t)$ are relatively prime as polynomials in $t$. One way to see this is to factor $D_A(t)$ and $t^n D_B(x/t)$ into linear factors over some extension field $G$ of $F$. Then the linear factors of $D_A(t)$ are all of the form $t-\alpha$ for $\alpha\in G$ while the linear factors of $t^n D_B(x/t)$ are all of the form $t-\beta x$ for $\beta\in G$. 

Since $A(x)$ and $B(x)$ are proper, it follows easily from \eqref{e-NB} that $A(t)B(x/t)$, as a rational function of $t$, is proper.
Thus $A(t)B(x/t)$ has a partial fraction expansion of the form 
\begin{equation*}
\frac{R_A(t,x)}{D_A(t)} + \frac{R_B(t,x)}{t^n D_B(x/t)},
\end{equation*}
where $R_A(t,x)$ and $R_B(t,x)$ are polynomials in $t$ with coefficients that are rational functions of $x$, and the degree of $R_B(t,x)$, as a polynomial in $t$, is less than $n$. Then $t^{-n}R_B(t,x)$ is a polynomial in $t^{-1}$ with no constant term, with coefficients that are rational functions of $x$, and $1/D_B(x/t)$, as an element of the Laurent series ring $F((t))((x))$,  has no positive powers of~$t$. Thus  $t^{-n}R_B(t,x)/D_B(x/t)$ has only negative powers of $t$. 

We know that $R_A(t,x)$ is a polynomial in $t$, so $R_A(t,x)/D_A(t)$ is a power series in $t$.  Thus $R_A(t,x)/D_A(t)$ is the sum of all terms in $A(t)B(x/t)$ with nonnegative powers of $t$, and so the constant term in $t$ in $A(t)/B(x/t)$ may be obtained by setting $t=0$ in $R_A(t,x)/D_A(t)$.
\end{proof}

While in principle the partial fraction method may be no more efficient than other methods, it may be easier in practice when using a computer algebra system  in which partial fraction expansion is already built in. If the denominator polynomials $D_A(x)$ and $D_B(x)$ are not irreducible, the built-in partial fraction function may give us a more refined decomposition than we need, but we can still apply the same method to this decomposition, or alternatively, we may use the extended Euclidean algorithm, or solve a system of equations by whatever method is convenient, to find the two-term partial fraction expansion.

We can compute binomial products in a similar way, using Prodinger's observation, discussed in Section \ref{s-cbc}, that the binomial product of $A(x)$ and $B(x)$ is the diagonal of \eqref{e-prod1}. It follows that $A(x)\odot B(x)$ is the constant term in $t$ in 
\begin{equation}
\label{e-ProdAB}
\frac{1}{1-t}A\left(\frac{x}{1-t}\right) B\left(\frac{x}{t}\right)
\end{equation}
which can also be computed by partial fraction expansion. (There is an analogue of Theorem \ref{t-pf} for binomial products.) 

For example, if we take $A(x)$ to be the Fibonacci generating function $x/(1-x-x^2)$ and $B(x)$ to be the Pell generating function $x/(1-x-2x^2)$ then expanding \eqref{e-ProdAB} by partial fractions in $t$ gives 
\begin{equation*}
\frac{1}{1-t}A\left(\frac{x}{1-t}\right) B\left(\frac{x}{t}\right)
   =\frac{x^2}{1-6x+7x^2+6x^3-9x^4}\bigl(R(x,t) + S(x,t)\bigr), 
\end{equation*}
where 
\begin{equation*}
R(x,t) = \frac{2-5x+x^2+3x^3 +(x-1)t}{1-x-x^2 + (x-2)t+t^2}
\end{equation*}
and
\begin{equation*}
S(x,t) = \frac{(1-x)t +2x^2 -3x^3}{t^2-2xt -x^2}= \frac{(1-x)t^{-1} +(2x^2 -3x^3)t^{-2}}{1-2xt^{-1} -x^2t^{-2}}.
\end{equation*}
Thus
\begin{equation*}
A(x) \odot B(x) = \frac{x^2}{1-6x+7x^2+6x^3-9x^4}R(x,0) = \frac{2x^2 -3x^3}{1-6x+7x^2+6x^3-9x^4},
\end{equation*}
as we found by a different method in Section \ref{s-comput}.

\section{Acknowledgment}
We would like to thank Alin Bostan for helpful suggestions and for bring references \cite{abm}, \cite{bfss}, and \cite{cmp} to our attention.

\bigskip
\hrule
\bigskip

\noindent 2020 {\it Mathematics Subject Classification}:
Primary 05A19, Secondary 05A15.

\noindent \emph{Keywords: }  binomial convolution, Hurwitz product, Hadamard product, resultant, rational power series.

\bigskip
\hrule
\bigskip

\noindent 
(Concerned with sequences
\seqnum{A000045},
\seqnum{A000073},
\seqnum{A000129},
\seqnum{A001045},
\seqnum{A001582},
\seqnum{A001608},
\seqnum{A002203},
\seqnum{A002203}, and
\seqnum{A014551}.)

\bigskip
\hrule
\bigskip

\end{document}